\documentclass[12pt]{article}
\usepackage[english]{babel}
\usepackage[colorlinks,hyperindex]{hyperref}

\newcommand{\dist}{\mbox{\rm dist}}

\newcommand{\Riem}{{\rm Riem}}

\newcommand{\Ricci}{ {\rm Ricci}}

\newcommand{\ep}{\varepsilon}

\newcommand{\grad}{\nabla}

\newcommand{\tIdist}{ {^{^{\ti I}}\! \dist }}
\newcommand{\Idist}{ {^{^I}\! \dist }}

\newcommand{\si}{\sigma}
\newcommand{\intersect}{\cap}

\newcommand{\lap}{\Delta}

\newcommand{\boundary}{\partial}

\newcommand{\on}{\over}

\newcommand{\ti}{\tilde}

\newcommand{\upI}{ {}^{^{ I}}\! }

\newcommand{\upgz}{^{^{g_0}}\!}

\newcommand{\upgt}{^{^{g(t)}}\!}
\newcommand{\uph}{^{^h}\!}

\newcommand{\upg}{^{^g}\!}

\newcommand{\partt}{ {\partial \on {\partial t}} }

\newcommand{\timess}{×}

\newcommand{\br}{ {\begin{rema}}}
\newcommand{\bl}{ {\begin{lemm}} }
\newcommand{\er} { {\end{rema}}}
\newcommand{\el}{{\end{lemm}} }
\newcommand{\bc}{  {\begin{coro}} }
\newcommand{\ec}{  {\end{coro}} }

\newcommand{\de}{\delta}

\newcommand{\lapg}{{\upg \lap}}

\usepackage{amsthm} %braucht man fuer swapnumber,... 
\usepackage{amsmath}% enth"alt amstext,amsbsy,amsopn siehe script mehr spalten als bei eqnarray 
\usepackage{amssymb} % f"ur mathematische Symbole der ams-Fonts
\usepackage{latexsym} 
\usepackage{array} 
\usepackage[T1]{fontenc} % 
\usepackage{fancyhdr} % seitenstil (strich,...)%
\usepackage{xspace} % 
\usepackage[dvips]{epsfig} % grafiken einbeziehen 
\usepackage{graphicx} % grafiken einbeziehen

\newtheorem{defi}{Definition}[section]

\newtheorem{prob}[defi]{Problem}
\newtheorem{rema}[defi]{Remark}

\newtheorem{theo}[defi]{Theorem}
\newtheorem{lemm}[defi]{Lemma}

\newtheorem{coro}[defi]{Corollary}

\begin{document}

\title{ Local results for flows whose speed or height satisfies a bound of the form $\frac c t$ }
\author{Miles Simon\\
Universit{ä}t Freiburg}
\maketitle 
%\vskip  0.05 true in
%\vskip  0.1 true in
 \numberwithin{equation}{section}
\numberwithin{defi}{section}

\thispagestyle{empty}
\mbox{}

\begin{abstract}

In this paper we prove local results for solutions to the Ricci flow (heat flow) whose speed (height) is bounded
by $\frac c t$ for some time interval $ t \in (0,T)$.
These results are contained in chapter 7 of \cite{HaSi}.
In \cite{HaSiPap1} further results from \cite{HaSi} may be found. In particular, there 
we construct short time solutions to Ricci flow
for a class of compact Riemannian manifolds with isolated 
conelike singularities. 
The resulting solutions satisfy a bound of this form (the speed is bounded by $\frac c t$ for some time interval
$t  \in (0,T)$).

\end{abstract}

\vskip  0.07 true in

\section{Local results for Ricci flow and reaction diffusion equations in general}

In \cite{Ha82}, R.Hamilton introduced the Ricci flow. A smooth family of metrics
$(M,g(t))_{t \in [0,T)}$ is a solution to the Ricci flow with initial value $g_0$, or is a 
Ricci flow of $g_0$ if
\begin{eqnarray}
&&\partt g(t)= - 2 \Ricci(g(t)) \ \forall \ t \in [0,T), \cr
&&g(0) = g_0.
\end{eqnarray}
Ricci flow has been extensively studied, and has led to many results in topology and geometry: see
for example \cite{BoWi},\cite{CCCY}, \cite{Ha82}
\cite{Ha86},
\cite{Ha95},
\cite{Ha95a},
\cite{Ha99},
\cite{Pe1},
\cite{Pe2}
(see \cite{CaZh}, \cite{KlLo} and \cite{MoTi} for good expositions of the works \cite{Pe1},\cite{Pe2}).

In Theorem 6.1 (and the proof thereof) of \cite{HaSiPap1}, 
we see that  for certain compact Riemannian manifolds with cone-like singularities, 
it is possible to obtain a short time solution 
to the Ricci flow, and also to 
obtain control over the curvature, volume and diameter for some well defined time interval.
This control on the curvature (curvature behaves like ${c \on t}$) is global, and suffices 
to prove the main theorem, Theorem 6.1. there.
We are, however, also interested in the local behaviour of such solutions
(see also \cite{ScScSi} and \cite{Si} for further examples of solutions to Ricci flow whose speed
is bounded by $ \frac{c}{t}$).

\vskip 0.1 true in

The following is taken directly from Chapter 7 of \cite{HaSi}).

\begin{prob}
For the class of solutions whose curvature is bounded by ${c \on t}$ on some time interval $[0,T)$: 
if a region  is diffeomorphic to a Euclidean ball and has bounded geometry at the curvature and $C^0$
level at time zero,
does some smaller region (contained in the initial region)
remain bounded geometrically at the curvature and $C^0$ level for a well controlled time interval?
\end{prob}

In the paper \cite{Pe1}, G.Perelman proved a local result  (see \cite{Pe1} Section 10) in the setting
that one has a solution to Ricci flow which is compact, has bounded curvature at each time, 
and for which the absolute value of the Riemannian curvature on $B(g_0)(x_0,r_0)$ is 
bounded by $ {1 \on r_0}$, and for which $B(g_0)(x_0,r_0)$ is close in some
$C^0$  sense to 
the Euclidean ball of radius $r_0$  (see Theorem 10.3 of  \cite{Pe1} for an exact statement of the
theorem).
There, the behaviour like $ c \on t$ is not assumed, but is in fact proved in a seperate theorem, Theorem 10.1
of \cite{Pe1}.
In Theorem 10.1 he proves that the curvature at a point $x_0$ behaves like ${ c \on t}$ 
(for $t \in [0,\de^2 r^2]$,  $\de = \de(n) >0$ ) under the
even weaker assumptions that the scalar curvature on  $B(g_0)(x_0,r_0)$ is bounded from below
by $-r^2_0$, and  $B(g_0)(x_0,r_0)$ is close in some $C^0$ sense to the Euclidean ball of radius $r_0$:
see theorem 10.1 of \cite{Pe1} for an exact statement of the theorem.

Let us define more precisely what we mean by:  
{\it the geometry is bounded on the curvature and $C^0$ level}.

\begin{defi}\label{boundedgeo}
The Euclidean $n$-ball of radius $r>0$ $ U = {\upI B_r(0)}$ with a Riemannian metric $g_0$ is 
geometrically bounded by $c$ if
\begin{eqnarray}
&&\sup_{U} {\upgz |} \Riem(g_0) | \leq {c \on r} \\
&& {c I} \leq g_0  \leq {1 \on c} I
\end{eqnarray}
where $I$ is the standard metric on ${\upI B_r(0)}$.
\end{defi}

Clearly the set is topologically the same as a ball (per definition).
Hence a Euclidean ball with a metric $g_0$ is {\it bounded geometrically} if the given Riemannian  metric
$g_0$ can be compared to the standard metric $I$ on the curvature and $C^0$ level.
 
In the following theorem we show that if $U$ is a Euclidean ball of radius one, and $(U,g_0)$ is 
bounded geometrically by some constant $c$,
and $(\bar U,g(t))_{t \in [0,T)}$ is a smooth solution to the Ricci flow
whose  curvature is  bounded above by ${c \on t}$, then, 
for a well controlled time interval,
a Euclidean ball of radius $r(c,n)>0$ (contained in the initial ball with the same middle point) with the metric $g(t)$
is bounded geometrically by $4c$.

\begin{theo}\label{pseudolemma}

Let $(U= {\upI B_1(0)},g_0)$ be geometrically bounded by $c_0$.
Define $${\upI \dist}(x) = \dist(I)(\boundary U ,x),$$ where the distance is taken with respect to the Euclidean metric $I.$
Let $( {\upI B_{1 + \de}(0)},g(t))_{t \in [0,T)}$ ($\de >0$) be any smooth
solution to the Ricci flow satisfying $g(0) = g_0$ and
\begin{eqnarray}
t{\upgt |}\Riem(x,t)| && \leq  c_0, \label{ctb} \\
\end{eqnarray} \label{pseudo}
 for all $x \in {\upI B_{1 + \de}(0)}$ and for all $t \in [0,T)$.
Then there exists an $N = N(c_0,n),$ such that
\begin{eqnarray*}
{\upgt |}\Riem(x,t)|\Idist^2(x) < N,
\end{eqnarray*}
for all $ x \in U $ with $\Idist^2(x) \geq Nt,$
and $t \leq T.$ 
In particular, $({\upI B_{1 \on 4N} (0)},g(t))$ is bounded geometrically by $e^1 c_0$
for all $t \leq  \min( {1 \on 4N},T).$ 
\end{theo}

\begin{proof} For the proof, $|\Riem(h)|$ will always refer to $ {\uph |}\Riem(h) |.$ 
Choose $N$ big, and 
assume that the theorem does not hold.
Then there must be a first time
$t_0 \in [0,\min(T , {1 \on 4N} ))$ and point $x_0 \in U$ 
where the theorem does not hold.
\begin{eqnarray*} 
&&\Idist^2(x_0)  \geq  N t_0, \cr
&& x_0 \in U, \cr
&&|\Riem(x_0,t_0)|\Idist^2(x_0) = N ,
\end{eqnarray*}
and 
\begin{eqnarray} 
|\Riem(x,t)| \Idist^2(x) \leq N \ \forall \ (x,t) \in U \timess [0,t_0]\  {\rm with } \Idist^2(x) \geq Nt \label{mainestimate}
\end{eqnarray} 

Let us rescale: ${\ti g}(x,s):= a g(x, {s \on a}),$
for $s \in [0, a t_0],$ and
$\ti I = a I.$
We define:
${\tIdist}^2(x) := \dist^2(\ti I)(\boundary U,x)$
(that is, the distance from the boundary of $U$ to the point $x$ with respect to the metric $\ti I$)
which gives us
${\tIdist}^2(x) =  a \Idist^2(x).$
Set $$a := { N \on  \Idist^2(x_0)},$$ ( $a$ is then bigger than $N$)
and for $t \in [0,t_0],$ let $\ti t := a t.$ 
Then ${\tIdist}^2(x_0) = N,$ and hence $N= |\Riem (x_0,t_0)| \Idist^2(x_0) = |\ti \Riem(x_0,\ti t_0)| {\tIdist}^2(x_0)
= |{\ti \Riem}(x_0, \ti t_0)|N,$
which implies that
$|{\ti \Riem}(x_0, \ti t_0)| = 1.$
Furthermore,
$ \ti{t_0} = { t_0 N \on  \Idist^2(x_0)} \leq 1,$
since $ \Idist^2(x_0)  \geq  N t_0.$
We also have
\begin{eqnarray}
 |{\ti \Riem}(x, \ti t)| {\tIdist}^2(x) 
 = |\Riem(x,t)|\Idist^2(x)  \leq  N  \label{Rdistbound}
\end{eqnarray}
for all $(x,\ti t) \in U\timess [0,\ti t_0]$ which satisfy 
$ {\tIdist}^2(x) \geq N{\ti t}$
in view of the facts 
\begin{itemize}
\item[{(i)}]
 ${\ti  \Idist}^2(x) \geq N{\ti t} \iff  {\Idist}^2(x) \geq N{t}$ and
\item[{(ii)}]
 $\ti t \leq {\ti t_0} \iff t \leq {t_0}$ and 
\item[{(iii)}]
equation \eqref{mainestimate} holds).
\end{itemize}
Now we consider two cases:
\begin{itemize}
\item{case 1: $  {\tIdist}^2(x_0) \geq 2N \ti t_0 (\iff \ti t_0 \leq {1 \on 2})$}
\item{case 2: $  {\tIdist}^2(x_0) < 2N \ti t_0 (\iff  \ti t_0 >   {1 \on 2})$}
\end{itemize}
Assume case one holds.
Then for $y$ satisfying $ {\tIdist}^2(y) \geq {N \on 2}$ we see that (use $\ti t_0 \leq {1 \on 2}$)
$$ {\tIdist}^2(y) \geq {N \ti t_0} \geq N \ti t $$
for all $\ti t \leq \ti t_0,$
and hence $$ |{\ti \Riem}(y, \ti t)|{\tIdist}^2(y)  \leq  N,$$
in view of \eqref{Rdistbound}. Hence, 
\begin{eqnarray}
|{\ti \Riem}(y, \ti t)| &\leq& 2 \ \forall  \ \ti t \leq \ti t_0, \label{riembound3}
\end{eqnarray}
in view of the assumption on $y$ (notice that $ {\tIdist}^2(x_0) = N \geq {N \on 2}$ 
and so $y=x_0$ is valid in \eqref{riembound3}: that is, 
$|{\ti \Riem}(x_0, \ti t)|  \leq  2$ for all 
$\ti t \leq \ti t_0$).
That is,
$$|{\ti \Riem}(y, \ti t)| \leq 2 \ \forall  \ \ti t \leq \ti t_0, \ \ \forall
y \in {\ti B}_{N \on 4}(x_0).$$
Using the fact that 
$$  {I \on c_0} \leq g_0 \leq c_0 I $$ we obtain
that 
$$ {\ti I \on c_0} \leq {\ti g}_0 \leq c_0 \ti I  $$
and $$ \sup_{ y \in {\ti B}_{N \on 4}(x_0)} | \Riem(\ti g_0)|(y) \leq {c_0 \on N},$$
since $\ti g_0 = a g_0,$ and $a = { N \on  \Idist^2(x_0)} \geq N$ and $(U,g_0) = ({}^I B_1(0),g_0)$ is geometrically
bounded by $c_0$.
We also have 
$|{\ti \Riem}(x_0,\ti t_0)| = 1.$
This contradicts Lemma \ref{riemsmall}, if $N = N(c_0,n)$ is chosen large enough.

So assume case two holds. This is equivalent to
${1 \on 2} < \ti t_0.$
Then, for $\ti t \leq {1 \on 2},$ we still have the
estimate $$ |{\ti \Riem}(y, \ti t)| \leq 2, \ \ \forall y \ \ \mbox{with} \ \ { \tIdist}^2(y) \geq {N \on 2}$$
(this may be seen as follows: $$ {\tIdist}^2(y) \geq {N \on 2} \geq N \ti t,$$ since 
$\ti t \leq {1 \on 2} < {\ti t}_0,$ and so, using
\eqref{Rdistbound}, we obtain the estimate).

For $\ti t_0 \geq \ti t > {1 \on 2},$
we have  $$ |{\ti \Riem}(y, \ti t)|  \leq {c_0 \on \ti t} \leq {2c_0},$$
in view of the assumption \eqref{ctb}.
W.l.o.g. $c_0\geq 1.$ Hence $$ |{\ti \Riem}(y, \ti t)| \leq {2c_0}, \forall t \in [{0},\ti t_0), \ \ \forall \ y \in \ti B_{N \on 4}(x_0).$$
Furthermore, $|{\ti \Riem}(x_0,\ti t_0)| = 1.$
Once again, this leads to a contradiction
for $N= N(c_0,n)$ chosen large enough, in view of Lemma \ref{riemsmall}.
(as in case 1, we have 
\begin{eqnarray} && 
    {\ti I \on c_0} \leq {\ti g}_0 \leq c_0 \ti I, \cr
 &&  \sup_{ {\ti B}_{N \on 4}(x_0)}  | Riem(\ti g_0)| \leq {c_0 \on N},
\end{eqnarray}
 and so we may apply the lemma).
\end{proof}

The following lemma, which  is  used to help prove the locality theorem
above, is a standard result from the theory of parabolic equations

 \begin{lemm}\label{riemsmall}
Let $(M,g(t))_{t \in [0,T)},$ be a complete smooth solution to Ricci flow and $(U,x_0,I),$ $U \subset M$
be isometric to  to the (open) Euclidean ball $B_{N}(0)$of radius $N,$ ($x_0 \sim 0$) and centre $0$.
Assume that:
\begin{eqnarray*}
&&  {I \on c_0} \leq g_0 \leq   c_0 I \cr
&&\sup_U |\Riem(\cdot,t)|  \leq  c_1 \  \forall \ t \in [0,T), \cr
&&\sup_U |\Riem|(\cdot,0) \leq \ep. \cr
\end{eqnarray*}
Then for all $\si^2 \geq \ep^2$ there exists $a =a(c_0,c_1,n)$ and $N_* = N_*(\si,c_0,c_1,n)> 0$
such that if $N \geq N_*$ then
\begin{eqnarray}
|\Riem|^2(x_0,t) \leq \si^2\exp^{at} \ \forall \ t \in [0,T)
\end{eqnarray}
In particular,
if $\ep \leq \exp^{-a} \de$ for some $\de>0$ then
(choose $\si = \exp^{-a} \de$: then $N_*$ is a constant depending only on $c_0,c_1,n,\de$ ) there is a
$N_* = N_*(c_0,c_1,n,\de) >0$ such that
$$|\Riem(x_0,t_0)| \leq {\de^2}, \ \forall \ t_0 \in [0,T) \intersect [0,1) $$
if $N \geq N_*$.
\end{lemm}

\begin{proof}
Clearly, in view of the conditions in the assumption, and the equation of evolution for the metric, we have
 $$  {1 \on c_2(c_0,c_1,n)}I \leq g(t)  \leq  c_2(c_0,c_1,n) I \ \forall  \ t \in [0,T)$$
Set 
$$f(\cdot,t):= |\Riem|^2(\cdot, t) - \si^2(1 + \rho^2)\exp^{at},$$ where $a$ is to be chosen.
%is less then zero. Choose $\si = \si(c_0,c_1) > 0 small$ and assume $N$ is so big, that
%$R(\cdot, t)- \si^2(1 + \rho^2)$ on the boundary of $B_{N \on 2 c^2_2}(0).$
and $\rho(x,t) := \dist(g(t))(x,x_0).$
Then
\begin{eqnarray*}
\partt f \leq&&  \lapg f + 4|\Riem|^3 - a \si^2(1 + \rho^2)\exp^{at} - 2 \si^2(\partt \rho) \rho\exp^{at}
+ \si^2 \exp^{at}\lapg(\rho^2) \cr
 \leq  && \lapg f + 4 c_1|\Riem|^2  - a \si^2(1 + \rho^2)\exp^{at} \cr
&& + 4 \si^2(n-1)c_1 \rho^2 \exp^{at}
+ c_1\si^2 \exp^{at}c(c_1,n) \rho  \cr
  \leq  && \lapg f + 4 c_1|\Riem|^2 - {a \on 2} \si^2(1 + \rho^2)\exp^{at} \cr 
\end{eqnarray*}  
for all $x \in B_{{N\on 2c^2_2}}(0)$
for appropriately chosen $a = a(n,c_1),$ where here we have used
the Hessian comparison principle in order to estimate $\lapg(\rho^2),$
and the fact that all distance minimising
geodesics (in terms of $g(t)$) between $0$ and  points in  ${^I B}_{{N\on 2c^2_2}}(0)$ must lie in ${^I B}_{N}(0)$
(this last fact may be seen as follows: the length of a ray from $0$ to 
$p \in \boundary B_{{N\on 2c^2_2}}(0)$ is trivially  bounded from above by 
$ c_2 { N \on 2 c^2_2} = { N \on 2 c_2 }.$  For any curve starting from $0$ which reaches the boundary of $B_{N}(0),$
the length is bounded from below by ${1 \on c_2}N$). 
Using the definition of $f$, we get 
\begin{eqnarray*}
\partt f & \leq & \lapg f + 4 c_1 f + 4 c_1\si^2(1 + \rho^2)\exp^{at} - {a \on 2} \si^2(1 + \rho^2)\exp^{at} < 
4 c_1 f,
\end{eqnarray*} 
for appropriately chosen $a = a(n,c_1).$
Now for 
$\si^2 \geq \ep^2$ 
we have $f(\cdot,0) < 0 $ on $B_{{N\on 2c^2_2}}(0).$
Now choose $N = N(\si, c_0,c_1,n)$ 
so large that  $f(\cdot,t) < 0$ on $\boundary B_{{N\on 2c^2_2}}(0)$ for all $t \in [0,T).$
The maximum principle then implies that
$f$ is less than zero for all $t \in [0,T),$
for all $x \in B_{{N\on 2c^2_2}}(0).$
Note that although $\rho^2$ is not smooth everywhere, using a trick of Calabi, we may still draw the same conclusion:
see the  proof of a Theorem 7.1 in \cite{Si} (essentially we define a new function $\ti \rho (x,t) = \rho (x,q,t) + \rho(q,x_0,t)$
for some appropriately chosen $q$ so that $\ti \rho$ is smooth in a small neighbourhood of $(p_0,t_0)$
where  $(p_0,t_0)$ is the first time and point where $f(p_0,t_0) = 0.$
Thn we define: 
$$\ti f(\cdot,t):= |\Riem|^2(\cdot, t) - \si^2(1 + \ti \rho^2)\exp^{at},$$
and argue with $\ti f:$
due to the triangle inequality we have $\ti f \leq f$ and hence
$ \ti f < 0$ for $ t \leq t_0.$
Furthermore $\ti f(x_0,t_0) = 0,$ as $q$ lies on a shortest geodesic between
$x_0$ and $p_0$ at time $t_0$.
Hence we may apply the maximum principle and still obtain a contradiction.
Hence  $f < 0.$
\end{proof}

So we see that if a local region is relatively well controlled at time zero, and the curvature behaves globally like
$c \on t$ near time zero, then we can show that a (well defined) smaller region remains well behaved for a 
a well defined time interval. 

\vskip 0.1 true in

The following remarks did not appear originally in \cite{HaSi}.

\begin{rema}
Notice that the condition $\Idist^2(x) \geq Nt$ can be removed, since if $\Idist^2(x) \leq Nt$
then $|\Riem(x,t)|\Idist^2(x) \leq |\Riem(x,t)|Nt \leq c_0N$ if the conditions of the theorem are satisfied.
\end{rema}
\begin{rema}
$\Idist^2(x) = (1 -|x|)^2.$
\end{rema}

\section{Local results for solutions $f$ to the heat equation which satisfy a bound of the form $f\leq \frac 1 t $. }
The following is taken directly from Chapter 7 of \cite{HaSi}.

In fact, the important bound which leads to the locality result (Theorem \ref{pseudolemma}), is the bound of the
form $|\Riem| \leq {c \on t}.$ The argument was a scaling argument which used the parabolic maximum principle, 
and didn't really have anything to do with Ricci-flow.
We illustrate this somewhat more precisely, by showing that a similar result holds for the heat flow.

\begin{lemm}\label{pseudoheatlemma}
Let $f$ be a smooth solution to the heat flow on $B_2(0)\timess [0,1]$ with 
\begin{eqnarray}
&&ft|_{B_1(0)\timess [0,1]}  \leq 1 \ \ \forall t \in [0,1], \label{tcon} \\ 
&&f(x,t)  \geq 0, \ \ \forall x \in B_2(0), t \in [0,1] \label{tcon2} \\
&&\sup_{{\bar B}_1(0)} f(x,0) (1 - |x|^2)^2  \leq 1   \label{tcon3} 
\end{eqnarray}
Then
\begin{eqnarray*}
 \sup_{{\bar B}_1(0)} f(x,t) ( (1 - |x|^2)^2 - 50nt)  & \leq 50n 
\end{eqnarray*}
for all $t \in [0,1].$
\end{lemm}

\begin{proof}
Set $$l(x,t) := f(x,t)( (1 - |x|^2)^2 - Mt)  - M,$$ where
$M = 50n$. 
Then $l \in C^{\infty}( \bar{B}_{3 \on 2 }(0)\timess [0,1]),$ and $l(x,0) \leq - 50,$ and
$l(\cdot,t) \leq -50n <0$ on
$\boundary(  {B}_{1}(0) )$ for all $t \in [0,1].$
Hence, if there is a time and point 
$(x,t) \in \bar{B}_{1}(0)\timess [0,1]$ where $l(x,t) \geq  0,$ then due to compactness, there must be a first time $t_0$ 
and point $x_0 \in B_1(0),$ where $l(x_0,t_0) = 0$ (possibly there is more than one point $x_0$, but there exists at least one
point $x_0$).
Assume that $(x_0,t_0)$ is such a point, and $(x_0,t_0)$ satisfies
$ (1 - |x_0|^2)^2 - M t_0 \leq {M \on 2} t_0.$ Then
\begin{eqnarray}
l(x_0,t_0) &&= f(x_0,t_0) ( (1 - |x_0|^2)^2 - M t_0) - M \cr
&& \leq  f(x_0,t_0) {M \on 2}t_0 - M \cr
&&\leq -{M \on 2} < 0,
\end{eqnarray}
(in view of the conditions \eqref{tcon} and  \eqref{tcon2} ) which contradicts the fact that $l(x_0,t_0) = 0.$
So we may assume, without loss of generality, that
\begin{eqnarray}
 (1 - |x_0|^2)^2 - M t_0 \geq {M \on 2} t_0 \label{pos}
\end{eqnarray}
Notice that this implies
$$ {2 \on 3}(1 - |x_0|^2)^2  \geq  M t_0 $$ which further implies that
\begin{eqnarray}
  (1 - |x_0|^2)^2 -Mt_0 \geq {1 \on 3}  (1 - |x_0|)^2 \label{schluessel}
\end{eqnarray}
 
We calculate the evolution equation for $l$ for such an $(x_0, t_0).$
\begin{eqnarray}  
\partt l =&&
\lap(l) - M f - f \lap( (1- |x|^2)^2) - 2 \grad_i f \grad_i (1 -|x|^2)^2 \cr
=&& \lap (l)  - M f -  f ( 8 |x|^2 - 4 n (1 -|x|^2) ) -2 \grad_i f \grad_i (1 -|x|^2)^2 \cr
\leq&&  \lap(l) - M f + 4n f  -2 \grad_i f \grad_i (1 -|x|^2)^2 \cr
\leq&&  \lap(l) - {M \on 2} f -2 \grad_i f \grad_i (1 -|x|^2)^2 \cr
=&& \lap(l) - {M \on 2} f -\frac{2 \grad_i( f(  (1 -|x|^2)^2 -Mt))  \grad_i (1 -|x|^2)^2 }{  ((1 -|x|^2)^2 -Mt)} \cr
&& + 2f  {\grad_i  (1 -|x|^2)^2  \grad_i (1 -|x|^2)^2 \on  ((1 -|x|^2)^2 -Mt)}  \cr
=&&  \lap(l) - {M \on 2} f -{2 \grad_i( f(  (1 -|x|^2)^2 -Mt))  \grad_i (1 -|x|^2)^2 \on  ((1 -|x|^2)^2 -Mt)} \cr
&& + 8 f   |x|^2 {(1 -|x|^2)^2  \on   ((1 -|x|^2)^2 -Mt)} 
\end{eqnarray}
and hence, in view of the inequality \eqref{schluessel},
\begin{eqnarray}  
{\partt l } 
&&\leq  \lap(l)  - {M \on 2} f-{2 \grad_i l \grad_i (1 -|x|^2)^2 \on  ((1 -|x|^2)^2 -Mt)}
+ 24 f   |x|^2  \cr
&& \leq   \lap(l) -{2 \grad_i l \grad_i (1 -|x|^2)^2 \on  ((1 -|x|^2)^2 -Mt)}
+ 24 f - {M \on 2} f  \cr
&& <    \lap(l) -{2 \grad_i l \grad_i (1 -|x|^2)^2 \on  ((1 -|x|^2)^2 -Mt)}.
\end{eqnarray}
As $l(x_0,t_0)$ is a local maximum, we obtain a contradiction
(note that  $((1 -|x_0|^2)^2 - Mt_0)  > 0 $ due to the assumption
\eqref{pos}).
\end{proof}

In fact a scaled version of
this lemma is true whenever we have at most polynomial growth of $f$ in $t,$ as the following lemma shows.

\begin{lemm}\label{pseudoheatlemmatwo}
Let $f$ be a smooth solution to the heat flow on $B_2(0)\timess [0,1]$ with 
\begin{eqnarray}
&&ft^p|_{ B_1(0)\timess [0,1]}  \leq 1 \label{tconp} \cr 
&&f(x,t)  \geq 0, \ \ \forall x \in B_2(0), t \in [0,1] \\
&&\sup_{x \in {\bar B}_1(0)} (f(x,0)+1)^{1 \on p} (1 - |x|^2)^2  \leq 2  
\end{eqnarray}
($p \geq 1$).
Then
\begin{eqnarray*}
 \sup_{x \in {\bar B}_1(0)} (f(x,t) +1)^{1 \on p} ( (1 - |x|^2)^2 - M(n,p)t)  & \leq c(n,p) \ \ \forall t \in [0,1].
\end{eqnarray*}
\end{lemm}

\begin{proof}

If $$ f t^p \leq 1 $$ then $$(1 + f)t^p \leq 2 \ \forall \ t \leq 1,$$
and hence $$(1 + f)^{1 \on p} t \leq 2.$$
Setting $$ \ti f := (1 + f)^{1 \on p},$$ we have 
\begin{eqnarray}
&& \ti f t \leq  2, \cr 
&& {\ti f}(x,0)(1 - |x|^2)^2 \leq 2. 
\end{eqnarray}   
Using the fact that $f$ solves the heat equation and $\ti f \geq 1$, we calculate
\begin{eqnarray}
\partt \ti f = \lap (\ti f) - {1 \on p}({1 \on p}-1) { | \grad \ti f |^2 \on \ti f}.
\end{eqnarray}
Define $$\ti l := {\ti f}(x,t)( (1 - |x|^2)^2 - Mt)  - M,$$ where
$M = M(n,p)$ is a constant to be chosen. 
Let $(x_0,t_0)$ be a first time and point where $\ti l =0.$
Arguing as above, we obtain
\begin{eqnarray}  
 \partt \ti l & < &   \lap(\ti l) -{2( \grad_i \ti l) V_i}  - {M \on 2} \ti f  - {1 \on p}({1 \on p}-1) ( (1 - |x|^2)^2 - Mt)
 { | \grad \ti f |^2 \on \ti f} \label{mainpeq}
\end{eqnarray}
where 
$$V_i = { \grad_i (1 -|x|^2)^2 \on  ((1 -|x|^2)^2 -Mt)}.$$
Arguing as in the previous lemma, we see that we may assume that  
$((1 - |x|^2)^2 - Mt) \geq {1 \on 3} (1 - |x|^2)^2,$ 
and so $V_i$ can be considered to be smooth in a small neighbourhood of 
 $(x_0,t_0)$.
But the last term in the above inequality can also be estimated, as the following calculation shows.
\begin{eqnarray}
{c(p) \on \ti f}((1 - |x|^2)^2 - Mt)   \grad_i \ti f \grad_i \ti f   =&&
{c(p)\on \ti f} \grad_i (( (1 - |x|^2)^2 - Mt)  \ti f) \grad_i \ti f  \cr
&& - c(p) \grad_i ( (1 - |x|^2)^2 - Mt) \grad_i \ti f.   \label{eqcp}
\end{eqnarray}
Remembering that we may assume that  
$((1 - |x|^2)^2 - Mt) \geq {1 \on 3} (1 - |x|^2)^2,$ we get
\begin{eqnarray}
\lefteqn{- c(p) \grad_i ((1 - |x|^2)^2 - Mt) \grad_i \ti f  } \cr 
= && -  {c(p) \on ( (1 - |x|^2)^2 -Mt) }  \grad_i ( (1 - |x|^2)^2 - Mt) \grad_i (\ti f ( (1 - |x|^2)^2 - Mt)) \cr
&& +  c(p)  {\ti f  \on ( (1 - |x|^2)^2 - Mt) }\grad_i (1 - |x|^2)^2  \grad_i (1 - |x|^2)^2  \cr
\leq && - {3 c(p) \on ( (1 - |x|^2)^2  }   \grad_i ( (1 - |x|^2)^2 - Mt) \grad_i (\ti f ( (1 - |x|^2)^2 - Mt)) \cr
&& + 32 c(p) \ti f 
\end{eqnarray}
Substituting this inequality into \eqref{eqcp}, we get
$$ {c(p) \on \ti f}((1 - |x|^2)^2 - Mt)   \grad_i \ti f \grad_i \ti f  \leq   
W_i \grad_i \ti l + 12 c(p) \ti f,$$
is a smooth vector field which is defined in a small neighbourhood (space and time) of the point $(x_0,t_0).$ 
Substituting this inequality into  equation \eqref{mainpeq}
we get 
\begin{eqnarray}  
 \partt (\ti l) & < &  \lap(\ti l) +  \grad_i (\ti l) Z_i. 
\end{eqnarray}
at the point $(x_0,t_0),$ where $Z$ is a smooth vector field which is defined in a small neighbourhood (space and time) of the point $(x_0,t_0).$ 
This gives  us a contradiction,
as $(x_0,t_0)$ is a local maximum for $\ti l.$

\end{proof}

\vskip 0.1 true in

\section*{Acknowledgements\markboth{Acknowledgements}{Acknowledgements}}

We would like to thank Peter Topping for helpful discussions
on Harmonic map heat flow and the Pseudolocality result of Perelman. 
Thanks to Klaus Ecker, Gerhard Huisken and Ernst Kuwert for their interest in this
work.

\vskip 0.1 true in
\rm

\vskip 0.3 true in
\centerline{Mathematisches Institut, Eckerstr. 1, 79104 Freiburg im Br., Germany }
 \centerline{e-mail: msimon at gmx.de }


\begin{thebibliography}{10}
\bibitem[1]{BoWi}
Böhm,C, Wilking,B
\newblock  Manifolds with positive curvature operators are space forms
\newblock 	arXiv:math/0606187v1 [math.DG], 2006

\bibitem[2]{CaZh}
Cao, Huai-Dong, and Zhu, Xi-Ping, 
\newblock  Hamilton-Perelman's Proof of the Poincaré Conjecture and the Geometrization Conjecture
\newblock arXiv:math/0612069
   

\bibitem[3]{CCCY} 
Cao,H.D., Chow, B., Chu,S.C., Yau,S.T.
\newblock Collected papers on the Ricci flow
\newblock Series in Geometry and Topology, Vol.37, International Press.




\bibitem[4]{Ha82}
Hamilton,~R.S.
\newblock Three manifolds with positive 
Ricci-curvature
\newblock {\em J.Differential Geom.},  17, no. 2, 255 -- 307, (1982).

\bibitem[5]{Ha86}
Hamilton,~R.S.
\newblock Four manifolds with positive curvature operator
 \newblock {\em J. Differential
Geom} 24 no. 2 , 153 -- 179, (1986).

\bibitem[6]{Ha93}
Hamilton,~R.S.
\newblock
 Eternal solutions to the Ricci flow,
\newblock {\em Journal of Diff. Geom.}, 38, 1-11 (1993)

\bibitem[7]{Ha95}
Hamilton,~R.S.
\newblock
 The formation of singularities in the Ricci flow,
\newblock {\em Collection: Surveys in differential geometry}, 
Vol. II (Cambridge, MA), 7--136, (1995).

\bibitem[8]{Ha95a}
Hamilton,~R.S.
\newblock A compactness property of the Ricci Flow
\newblock {\em American Journal of Mathematics}, 117, 545--572, (1995)


\bibitem[9]{Ha99}
Hamilton,~R.S.
\newblock Non-Singular solutions of the Ricci Flow on Three-Manifolds
\newblock {\em Comm. Anal. Geom.}, vol 7., no. 4, 695--729, (1999)

\bibitem[10]{KlLo}
Kleiner,B., Lott,J.
\newblock  Notes on Perelman's papers,
\newblock  arXiv:math/0605667

\bibitem[11]{MoTi} 
Morgan, J., Tian, G.-T.
\newblock Ricci Flow and the Poincare' Conjeture
\newblock arXiv:math/0607607

\bibitem[12]{Pe1}
Perelman,G.,
\newblock 
The entropy formula for the Ricci flow and its geometric applications
\newblock MarthArxiv link: math.DG/0211159

\bibitem[13]{Pe2}
Perelman,G.,
\newblock 
Ricci flow with surgery on three manifolds

\bibitem[14]{ScScSi}
Schnürer, O, Schulze, F., Simon, M.
\newblock Stability of Euclidean space under Ricci flow
\newblock  arXiv:0706.0421

\bibitem[15]{Si} 
Simon, M.,
\newblock Deformation of $C^0$ Riemannian metrics in the
direction of their Ricci curvature,
\newblock {\em Comm. Anal. Geom.},
10, no. 5, 1033-1074, (2002)

\bibitem[16]{HaSi}
Simon, M.
\newblock Habilitation Thesis:``Ricci flow of almost non-negatively curved three manifolds''
\newblock Freiburg University, Germany, 2006
\newblock available at http://home.mathematik.uni-freiburg.de/msimon/

\bibitem[17]{HaSiPap1}
Simon, M.
\newblock  Ricci flow of almost non-negatively curved three manifolds
\newblock arXiv:math/0612095 , 2006



\end{thebibliography}
\end{document}